\documentclass{amsart}

\usepackage{amssymb}
\usepackage{amsfonts}
\usepackage{amsmath}
\usepackage{enumerate}

\begingroup

\newtheorem{thm}{Theorem\hskip 5mm}[section]

\newtheorem{theorem}{Theorem\hskip 5mm}[section]

\newtheorem{prop}[thm]{Proposition\hskip 5mm}

\newtheorem{cor}[thm]{Corollary\hskip 5mm}

\newtheorem{lem}[thm]{Lemma\hskip 5mm}

\newtheorem{exa}[thm]{Example\hskip 5mm}

\newtheorem{note}[thm]{Note\hskip 5mm}

\endgroup

\def\r{{\mathfrak r}}
\def\U{{\mathrm U}}
\def\C{{\bf C}}
\def\Z{{\bf Z}}
\def\Sp{{\mathrm {Sp}}}

\def\l{{\lambda}}

\def\GL{{\mathrm{ GL}}}
\def\U{{\mathrm{ U}}}

\def\m{{\mathfrak m}}
\def\i{{\mathfrak i}}
\def\r{{\mathfrak r}}

\begin{document}

\title[Weil representations]
{The Weil representation of a unitary group associated to a ramified
quadratic extension of a finite local ring}

\author{Allen Herman}
\address{Department of Mathematics and Statistics, Univeristy of Regina, Canada}
\email{Allen.Herman@uregina.ca}
\thanks{Both authors were supported in part by NSERC Discovery Grants.}

\author{Fernando Szechtman}
\address{Department of Mathematics and Statistics, Univeristy of Regina, Canada}
\email{fernando.szechtman@gmail.com}



\begin{abstract}
We find all irreducible constituents of the Weil representation of
a unitary group $U_m(A)$ of rank $m$ associated to a ramified
quadratic extension $A$ of a finite, commutative, local and
principal ring $R$ of odd characteristic. We show that this Weil
representation is multiplicity free with monomial irreducible
constituents. We also find the number of these constituents and
describe them in terms of Clifford theory with respect to a
congruence subgroup. We find all character degrees in the special
case when $R$ is a field.
\end{abstract}

\maketitle

\section{Introduction}

Let $K$ be a non-archimidean local field with ring of integers
$\mathcal{O}$, maximal ideal $\mathfrak p$ and residue field
$F_q=\mathcal{O}/\mathfrak p$ of odd characteristic $p$. Let
$m\geq 1$ and consider the Weil representation $W:\Sp_{2m}(K)\to
\mathrm{GL}(\mathcal{V})$ of the symplectic group $\Sp_{2m}(K)$.
This is a projective representation defined over an infinite
dimensional complex vector space~$\mathcal{V}$. The restriction of
$W$ to the maximal compact subgroup $\Sp_{2m}(\mathcal{O})$ is an
ordinary representation whose decomposition into irreducible
constituents was determined by Prasad \cite{16}. These
constituents are the trivial module together with two irreducible
modules for each factor group $\Sp_{2m}(\mathcal{O}/{\mathfrak
{p}}^{2\ell})$, as $\ell$ runs over all positive integers, which
are precisely the two irreducible constituents of the top layer of
the Weil representation of $\Sp_{2m}(\mathcal{O}/\mathfrak
p^{2\ell})$ constructed in \cite{1} using different methods. While
Prasad's decomposition does not involve odd powers of $\mathfrak
p$, Cliff and McNeilly \cite{17} obtained an analogous
decomposition that only involves odd powers of $\mathfrak p$ by
restricting $W$ to another maximal compact subgroup of
$\Sp_{2m}(K)$ and, again, the resulting constituents are those
appearing in~\cite{1}.

Let $F$ be a quadratic extension of $K$, with ring of integers
$\mathcal{R}$ and maximal ideal~$\mathcal{P}$. Let $U_m(F)$ be a
unitary group of rank $m$ associated to $F/K$. We may imbed
$U_m(F)$ in $\Sp_{2m}(K)$ and consider the restriction of $W$ to
$U_m(F)$. This raises the question as to what is the decomposition of
the Weil representation of $U_m(F)$ when restricted to a maximal
compact subgroup.  This question is what motivates us to consider the decomposition problem
for the Weil representation of unitary groups $U_m(\mathcal{R}/\mathcal{P}^\ell)$ for $\ell\geq 1$.  Our problem depends in an essential manner on whether the extension $F/K$ is ramified or unramified.

In the unramified case let $A=\mathcal{R}/\mathcal{P}^\ell$ and $R=\mathcal{O}/\mathfrak {p}^\ell$.
Then $A=R[\sqrt{x}]$, where $x$ is a non-square
unit of $R$. When $\ell=1$ then $A=F_{q^2}$, which is the classical
case solved by G\'{e}rardin \cite{6}. The case of arbitrary
$\ell$ was solved by Szechtman \cite{3}.

The ramified case remains open and constitutes the main goal of
this paper. In this case let $A=\mathcal{R}/\mathcal{P}^{2\ell}$
and $R=\mathcal{O}/\mathfrak {p}^\ell$. Then $A=R[\sqrt{x}]$,
where $Rx$ is the maximal ideal of $R$. We find all irreducible
constituents of the Weil representation of $U_m(A)$. The Weil
representation is multiplicity free and all its irreducible
constituents are monomial. A description of each constituent in
terms of Clifford theory with respect to a suitable normal abelian
subgroup is given. We also find the total number of irreducible
constituents of the Weil module, while their degrees are given
only in the case $\ell=1$, i.e. when $A=F_q[t]/(t^2)$. The
calculation of the degrees in the general case is equivalent to
finding the index in a unitary group $U_m(B)$, where $B$ is
quotient of $A$, of the stabilizer of a basis vector of arbitrary
length. This is a non-trivial problem, specially for non-unit
lengths, whose solution appears in \cite{15}.

What makes the ramified case interesting and more complicated is
that the norm map $A^{\times}\to R^{\times}$ is not surjective, unlike what
happens in the unramified case. To exemplify one of the
consequences of a defective norm map we mention this: while in the
unramified case there is only one non-degenerate hermitian form of
rank $m$ over $A$ up to equivalence, in the ramified case there
are two of these and, more importantly, {\it the number of
irreducible constituents of the Weil module as well as their
degrees depend on the choice of the actual form; these character
degrees depend on the underlying form even when $m$ is odd, when
the isomorphism type of $U_m(A)$ is uniquely determined}. Nothing
like this happens in the symplectic or unramified unitary cases.

Interest in the Weil representation shows no sign of slowing
down. Since Weil's original paper \cite{12} dealing with classical
groups over local fields, several authors have directed attention
to the subject. Howe \cite{13} looked at the context in which the
Weil representation appears, namely through the action of the
symplectic group on the Heisenberg group, with the aim of
generalizing it to other groups. Shortly afterwards G\'{e}rardin
\cite{6} defined Weil representations for general linear,
symplectic  and unitary groups over finite fields and found their
irreducible constituents. He also gave an explicit character
formula in the unitary case.  Gow \cite{19} found the character fields and Schur indices of the irreducible constituents of the Weil representation of the symplectic group in the finite field case.
Tiep and Zalesskii \cite{18} gave various characterizations of the
Weil representation of symplectic and unitary groups in terms of
its restriction to standard subgroups.  Recently, Thomas \cite{20} gave a fairly explicit formula for the character of the Weil representation of the symplectic group in the finite field case.

The Weil representation of the special unitary group $SU_2(T)$,
where $T$ is a fairly general finite ring with involution, is
considered in a series of papers by Guti\'{e}rrez, Pantoja, and
Soto-Andrade. Much of their work is centered around finding a
Bruhat presentation for $SU_2(T)$. Once this is achieved they
construct a generalized Weil representation \cite{4} by assigning
to each Bruhat generator a linear operator, analogous to the Weil
operators arising in the classical case, and then verify that
the defining relations for $SU_2(T)$ hold.

Our paper is organized as follows. \S\ref{HI} recalls the
definition of Weil representation in a context that is
sufficiently general to apply to symplectic, unitary and other
groups. \S\ref{iur} describes a general method to imbed a unitary
group into a symplectic group, provides numerous examples of this
phenomenon, and shows that the Weil representation is compatible
with this type of imbedding. More precisely, one may restrict the
Weil representation of a symplectic group to a unitary group or
consider a Weil representation of the unitary group directly: they
will agree up to a linear character. \S\ref{con} constructs a Weil
module $X$ for $U_m(A)$ that takes advantage of the special
features of the ring $A$. \S\ref{secsta} lays the foundation for
an analysis of the Clifford theory of $X$ with respect to a
specific congruence subgroup of $U_m(A)$. \S\ref{sectop}
determines and describes in terms of Clifford theory all
irreducible constituents of the top layer of $X$, while
\S\ref{seccount} finds the exact number of them. \S\ref{secbot}
shows that the bottom layer of $X$ is trivial if $\ell=1$ or a
Weil module for a unitary group $U_m(B)$ over a quotient ring $B$
of $A$; this allows us, in \S\ref{secfull}, to obtain the full
decomposition of $X$ by repeatedly peeling off one top layer after
the other. \S\ref{secdeg} refers to the degrees of the components
of the top layer of $X$. Finally \S\ref{seclast} considers the
special case $\ell=1$, revisits the decomposition of $X$, and
explicitly displays all character degrees.

\section{The Schr\"{o}dinger and Weil representations}
\label{HI}

Before attempting to decompose the restriction to a unitary group
of the Weil representation of a symplectic group it will be well
to recall the definition of Weil representation, as well as the
auxiliary Schr$\mathrm{\ddot{o}}$dinger representation. We will do
this in a fairly general setting, to take advantage of this
flexibility in \S\ref{con}.

Rather than beginning, as in \cite{1}, with a non-degenerate,
alternating bilinear form $f:V\times V\to R$, where $R$ is finite
ring of odd characteristic and $V$ is a finitely generated
$R$-module, we will merely assume here that $R$ and $V$ are finite
additive abelian groups of odd order and that $f:V\times V\to R$
is a bi-additive function, in the sense that
$$
f(u+v,w)=f(u,w)+f(v,w),\; f(u,v+w)=f(u,v)+f(u,w),\quad u,v,w\in V.
$$
Associated to $f$ we have the isometry group $G=G(f)$ and
Heisenberg group $H=H(f)$.
The group $G$ consists of all group automorphisms of $V$ that preserve $f$,
i.e.
$$
G=\{g\in \mathrm{Aut}(V)\,\vert\, f(gu,gv)=f(u,v)\text{ for all }u,v\in V\}.
$$
The group $H$ is comprised of all pairs $(r,v)$, $r\in R$ and $v\in V$, with multiplication
$$
(r,v)(s,w)=(r+s+f(v,w),v+w).
$$
The associative law follows from the fact that $f$ is bi-additive.
Moreover, the neutral element of $H$ is $(0,0)$, and the inverse
of $(r,v)$ is
$$
(r,v)^{-1}=(-r+f(v,v),-v).
$$
Thus $H$ is indeed a group. Conjugation in $H$ is given by
\begin{equation}
\label{conj}
(r,v)(s,w)(r,v)^{-1}=(s,w)(f(v,w)-f(w,v),0).
\end{equation}
In particular, $(R,0)$ is a central subgroup of $H$, and it will
be identified with $R^+$.

Clearly $G$ acts by automorphisms on $H$ as follows:
$$
^g(r,v)=(r,gv).
$$

Since $R$ is a finite additive abelian group of odd order,
multiplication by 2 is an automorphism of $R$, whose inverse will
be denoted as multiplication by~$1/2$. We thus define the
alternating and symmetric parts of $f$ as follows:
$$f^-(u,v)=(f(u,v)-f(v,u)))/2,\quad f^+(u,v)=(f(u,v)+f(v,u)))/2,\quad u,v\in V.$$
Thus
$$
f=f^-+f^+\text{ and }G(f)=G(f^-)\cap G(f^+).
$$

For $g\in G$ and a representation $T:H\to\GL(X)$, the conjugate representation $T^g:H\to\GL(X)$ is defined by
$T^g(h)=T({}^g h)$ for all $h\in H$. We say that $T$ is $G$-invariant if $T$ is equivalent to $T^g$ for all $g\in G$.

We fix a linear character $\l:R^+\to \C^{\times}$ for the remainder of
this section. By a Schr\"{o}dinger character of type
$\l$ we mean a $G$-invariant irreducible character of $H$ whose
restriction to $R^+$ is a multiple of $\l$.

\begin{thm}
\label{sch} There is one and only one irreducible character of
$H(f)$ lying over $\l$ and invariant under the central involution,
say $\iota$, of $G(f)$, given by~$v\mapsto -v$. \\Moreover, this
character is actually $G(f)$-invariant, that is, it is a
Schr${\ddot{o}}$dinger character of $H(f)$ of type $\l$. Its degree
is equal to $\sqrt{|V|}/\sqrt{|V(\l)|}$, where
$$
 V(\l)=\{v\in V\,|\, \l(f(v,u)-f(u,v))=1\text{ for
all }u\in V\}=\{v\in V\,|\, \l(f^-(v,V))=1\}.
$$
\end{thm}

\noindent{\it Proof.} We first consider the case $f=f^-$. Let $V_1$ be a subgroup of $V$ maximal subject to
$$
\l(f(v,u))=1,\quad v,u\in V_1.
$$
We extend $\l$ to a linear character $\rho$ of the normal subgroup
$(R,V_1)$ of $H$ by setting $\rho(r,u)=\l(r)$. Let $(s,v)$ belong to the
stabilizer of $\rho$ in $H$. Since $f=f^-$ and $R$ has odd order,
it follows from (\ref{conj}) that $\l(f(v,V_1))=1$. From the
maximality of $U$ and the fact that $f(v,v)=f^-(v,v)=0$ we see
that $v\in V_1$, so $(s,v)\in (R,V_1)$. By Clifford theory,
$\chi=ind_{(R,V_1)}^H \rho$ is an irreducible character of $H$ lying
over $\l$. The fact that $\chi$ is $G(f)$-invariant and that no
other $\iota$-invariant irreducible character of $H$ lies over
$\l$ follows as in the proof of Proposition 2.1 of \cite{1}.

We move to the general case. Consider the map
$\beta:H(f)\to H(f^-)$ given by
\begin{equation}
\label{compt0}
\beta(r,v)=(r-f(v,v)/2,v).
\end{equation}
We claim that $\beta$ is a group isomorphism.  Indeed, if $(r,v)$,
$(s,w) \in H(f)$ then
$$\begin{array}{rcl}
\beta[(r,v)(s,w)] &=& \beta[(r+s+f(v,w),v+w)] \\
                          &=& (r+s+f(v,w)-f(v+w,v+w)/2,v+w) \\
                          &=& (r+s-f(v,v)/2-f(w,w)/2+f^-(v,w), v+w) \\
                          &=& (r -f(v,v)/2,v) (s -f(w,w)/2,w) \\
                          &=& \beta(r,v) \beta(s,w),
                  \end{array}$$
so $\beta$ is a group homomorphism. Since $\beta(r,v)=(0,0)$
forces $(r,v)=(0,0)$ and $\beta(r+f(v,v)/2,v)=(r-f(v,v)/2,v)$, our
claim is established. This isomorphism is compatible with the
actions of $G(f)$ on $H(f)$ and $H(f^-)$, in the sense that
\begin{equation}
\label{compt}
{}^g \beta(k)=\beta({}^g k),\quad k\in H(f).
\end{equation}
Let $T:H(f^-)\to\GL(X)$ be a Schr$\mathrm{\ddot{o}}$dinger representation of $H(f^-)$ of type $\l$. It then follows from (\ref{compt}) and the fact
that $G(f)\subseteq G(f^-)$ that $S=T\circ \beta:H(f)\to\GL(X)$
is a Schr$\mathrm{\ddot{o}}$dinger representation of $H(f)$ of type $\l$. Moreover, let $S_1:H(f)\to\GL(X_1)$ and $S_2:H(f)\to\GL(X_2)$ be $\iota$-invariant irreducible representations lying over $\l$. Then $T_1=S_1\circ\beta^{-1}$ and $T_2=S_2\circ\beta^{-1}$ are $\iota$-invariant irreducible representations of $H(f^-)$ lying over $\l$. By the first case, $T_1$ and $T_2$ are equivalent, whence $S_1$ and $S_2$ are equivalent.
\qed

\begin{cor}
\label{corsch}
 Suppose $V(\l)=0$. Then there is one and only one irreducible character of $H(f)$ lying over $\l$, namely
the Schr$\mathrm{\ddot{o}}$dinger character of type $\l$.
\end{cor}

\noindent{\it Proof.} In this case the
Schr$\mathrm{\ddot{o}}$dinger character of type $\l$ is has degree
$\sqrt{[H(f):R^+]}$, so is fully ramified (see \cite{10}, Exercise
6.3) and uniqueness follows.\qed

Let $S:H(f)\to\GL(X)$ be a Schr$\mathrm{\ddot{o}}$dinger representation of type $\l$. A Weil representation
of type $\l$ is a representation $W:G(f)\to\GL(X)$ satisfying:
$$
W(g)S(h)W(g)^{-1}=S({}^g h),\quad g\in G(f),h\in H(f).
$$

\begin{thm}
\label{w}
 There is one and only one Weil representation of type~$\l$, up to equivalence and multiplication by
a linear character of $G(f)$.
\end{thm}

\noindent{\it Proof.} Let $S:H(f)\to\GL(X)$ be a Schr$\mathrm{\ddot{o}}$dinger representation of $H(f)$ of type $\l$.
By Theorem \ref{sch}, given $g\in G$ there exists $P(g)\in\GL(X)$ such that
$$
P(g)S(h)P(g)^{-1}=S({}^g h),\quad g\in G(f),h\in H(f).
$$
By Schur's lemma, $P(g)$ is unique up to multiplication by a non-zero scalar.
By uniqueness, given $g_1,g_2\in G(f)$ there exists $f(g_1,g_2)\in \C^{\times}$ such that
$$P(g_1)P(g_2) = f(g_1,g_2)P(g_1g_2), \text{ for all } g_1,g_2 \in G(f).$$
Given $g\in G(f)$ we must find $c(g)\in\C^{\times}$ such that
$W:G(f)\to\GL(X)$, given by $W(g)=P(g)c(g)$, is a representation.
This can be achieved much as in the proof of Theorem 3.1 of
\cite{1}.\qed

In regards to the relationship between the Weil representations of $G(f)$ and $G(f^-)$ we have the following result.

\begin{prop}
\label{x1}
 Let $W:G(f^-)\to\GL(X)$
be a Weil representation of of type $\l$. Then its restriction to $G(f)$ is a Weil representation of type $\l$.
\end{prop}

\noindent{\it Proof.} Let $T:H(f^-)\to\GL(X)$ be a
Schr$\mathrm{\ddot{o}}$dinger representation of $H(f^-)$ of type
$\l$ and let $W:G(f^-)\to\GL(X)$ be an associated Weil
representation. Thus
$$
W(g)T(h)W(g)^{-1}=T({}^g h),\quad g\in G(f^-),h\in H(f^-).
$$
Appealing to the isomorphism (\ref{compt0}) and the compatibility condition (\ref{compt}) we may write this as follows:
$$
W(g)T(\beta(k))W(g)^{-1}=T({}^g \beta(k))=T(\beta({}^g k)),\quad g\in G(f^-),k\in H(f).
$$
Letting $S=T\circ\beta:H(f)\to\GL(X)$ and using $G(f)\subseteq G(f^-)$, we see that
$$
W(g)S(k)W(g)^{-1}=S({}^g k),\quad g\in G(f),k\in H(f).
$$
Thus $S$ is Schr$\mathrm{\ddot{o}}$dinger of type $\l$ and the proof is complete.\qed

Suppose for the remainder of the paper that $R$ is a finite,
principal, commutative local ring of odd prime characteristic $p$.
Since $2$ is invertible in the subring $\Z/p\Z$ of $R$ we see that
$2\in R^{\times}$, the unit group of $R$.  Let $\m$ be the maximal ideal
of $R$ and let $F_q=R/\m$ be the residue field of $R$, where $q$
is a power of $p$. The nilpotency degree of $\m$ will be denoted
by $\ell\geq 1$.

For the remainder of this section we let $V$ be a non-zero free
$R$-module $V$ and  $f:V\times V\to R$ an alternating bilinear
form that is non-degenerate, in the sense that the associated
linear map $V\to V^*$, given by $v\mapsto f(v,-)$, is an
isomorphism.

\begin{lem} The $R$-module $V$ has a basis $\{u_1,\dots,u_m,v_1,\dots,v_m\}$ such that
$$
f(u_i,u_j)=0=f(v_i,v_j),\; f(u_i,v_j)=\delta_{ij},\quad 1\leq
i,j\leq m.
$$
\end{lem}

\noindent{\it Proof.} Since $f$ is a non-degenerate alternating
form, we must have $\mathrm{rank}\, V>1$. Let $w_1,w_2,\dots,w_m$
be a basis of $V$. If $f(w_1, w_i)\in\m$ for all $i>1$
then~$f(\m^{\ell-1} w_1,V)=0$, a contradiction. We may assume
without loss of generality that $f(w_1,w_2)=1$. Set
$z_i=w_i-f(w_i,w_2)w_1+f(w_i, w_1)w_2$ for $i>2$. Then
$w_1,w_2,z_3,\dots,z_m$ is a basis of $V$ and
$f(w_1,z_j)=0=f(w_2,z_j)$. Thus the restriction of $f$ to the span
of $z_3,\dots,z_m$ is non-degenerate and the result follows by
induction.\qed

\medskip
It follows that $\mathrm{rank}\, V = 2m$ is even, and all
non-degenerate alternating bilinear forms on $V$ are equivalent.
Thus, the corresponding isometry groups are similar in $\GL(V)$.
The symplectic group associated to $f$ is
$$
\Sp_{2m}(R)=G(f)\cap \GL(V).
$$

A linear character $\l:R^+\to\C^{\times}$ is said to be primitive
if its kernel contains no non-zero
ideals of $R$. Such a linear character exists because $R$ has a
unique non-zero minimal ideal, namely $\m^{\ell-1}$.

\begin{cor} A Weil representation of $\Sp_{2m}(R)$ of primitive type $\l$ has degree $|R|^m$.
\end{cor}

\begin{thm}
\label{comega} Let $\widehat{V}$ be the permutation module
associated to $V$. Let $S:H(f)\to\GL(X)$ be a
Schr$\mathrm{\ddot{o}}$dinger representation of $H(f)$ of
primitive type $\l$ and let $W:\Sp_{2m}(R)\to \GL(X)$ be an
associated Weil representation. Then
$\widehat{V}\cong\mathrm{End}(X)$ as $\Sp_{2m}(R)$-modules. In
particular, if $G$ be a subgroup of $\Sp_{2m}(R)$, $\Omega$ is the
Weil character of the $G$-module $X$ and $O_G(V)$ stands for the
number of orbits of $G$ acting on $V$, then
$[\Omega,\Omega]=O_G(V)$.
\end{thm}

\noindent{\it Proof.} The first assertion is proven in Theorem 4.5 of \cite{1}. As for the second, we have
$\widehat{V}\cong\mathrm{End}(X)$ as $G$-modules, so the trivial $G$-module appears an equal number of times in both
of them, that is, $[\Omega,\Omega]=[\Omega\overline{\Omega},1_G]=O_G(V)$.\qed

\section{Imbedding unitary groups in symplectic groups}
\label{iur}

In this section we furnish a generic method to imbed a unitary group
inside a symplectic group, exhibit various examples of this nature
and, more importantly, prove that the Weil representation, as
defined in \S\ref{HI}, is compatible with this type of imbedding.
Thus the restriction to a unitary group of the Weil representation
of a symplectic group can be studied by looking directly at the
Weil representation of this unitary group, which is what we do
from \S\ref{con} onwards.

 We make the following assumptions throughout this section:
 $A$ is an associative $R$-algebra with identity endowed with an involution $*$ that fixes $R$ elementwise;
 $A$ is a free $R$-module of finite rank $s$; $V$ is a non-zero free right $A$-module of finite rank $m$;
 $h:V\times V\to A$ is
hermitian or skew-hermitian form relative to $*$, which means that
$h$ is $A$-linear in the second variable and satisfies
$$h(v,u)=\varepsilon h(u,v)^*,\quad u,v\in V,
$$
where $\varepsilon=-1$ in the skew-hermitian case and
$\varepsilon=1$ in the hermitian case; $h$ is non-degenerate,
i.e., the map $V\to V^*$, given by $v\mapsto h(v,-)$, is an
isomorphism of right $A$-modules, where $V^*$ is a right
$A$-module via
$$
(\alpha a)(v)=a^* \alpha(v),\quad v\in V,a\in A,\alpha\in V^*.
$$

\begin{lem}
\label{easref2} Let $d:A\to R$ be an $R$-linear map satisfying:

{\rm(C1)} If $a\in A$ and $d(ab)=0$ for all $b\in A$ then $a=0$.

{\rm(C2)} $d(a+\varepsilon a^*)=0$ for all $a\in A$.

\noindent Then the map $f:V\times V\to R$, given by
$$
 f(u,v)=d(h(u,v)),\quad u,v\in V
$$
is a non-degenerate alternating $R$-bilinear form and $G(h)$ is a
subgroup of $G(f)$. Consequently the unitary group
$$
\U_m(A)=G(h)\cap \GL_A(V)
$$
is a subgroup of the symplectic group
$$
\Sp_{sm}(R)=G(f)\cap \GL_R(V).
$$
\end{lem}

\noindent{\it Proof.} It is clear that $f$ is
$R$-bilinear and automatic that
$$
G(h)\subseteq G(f).
$$
We next claim that $f$ is alternating. Indeed,
let $u\in V$. Then
$$
2h(u,u)=h(u,u)+h(u,u)=h(u,u)+\varepsilon h(u,u)^*,
$$
so by (C2)
$$2f(u,u)=d(2h(u,u))=0.$$

We finally claim that $f$ is non-degenerate. Indeed, let $0\neq
u\in V$. Since $h$ is non-degenerate there is $v\in V$ such that $a=h(u,v)\neq 0$. By
(C1) there is $b\in A$ such that $d(ab)\neq 0$, whence
$$
f(u,vb)=d(h(u,vb))=d(h(u,v)b)=d(ab)\neq 0.
$$
Thus $f$ induces a monomorphism $V\to V^*$, and hence an
isomorphism since $|V|=|V^*|$.\qed

\medskip

It is natural at this point to ask about the relationship
between the Weil representations of $\Sp_{sm}(R)$ and $\U_m(A)$.

\begin{thm}
\label{resu}
 Let  $\l:R^+\to\C^{\times}$ be a primitive linear character
and let $\mu:A^+\to\C^{\times}$ be defined by $\mu(a)=\l(d(a))$ for $a\in
A$. Assume the hypotheses of Lemma \ref{easref2} hold. Then $\mu$ is
primitive, in the sense that the kernel of $\mu$ contains no
non-zero right ideals. Moreover, the restriction to $\U_m(A)$ of a
Weil representation of $\Sp_{sm}(R)$ of type $\l$ is a Weil
representation of  $\U_m(A)$ of type~$\mu$.
\end{thm}

\noindent{\it Proof.} Let $I$ be a right ideal of $A$ such that
$\mu(I)=1$ and let $a\in I$. Then
$$\l(d(a)r)=\l(d(ar))=\mu(ar)=1,\quad r\in R,$$
whence $d(a)=0$ by the primitivity of $\l$. Therefore $d(I)=0$, so
$I=0$ by (C1). Thus $\mu$ is primitive.

We have a group homomorphism $\beta:H(h)\to H(f)$ given by
$$
\beta(a,v)=(d(a),v),\quad a\in A,v\in V.
$$
Let $T:H(f)\to\GL(X)$ be a Schr$\mathrm{\ddot{o}}$dinger representation of $H(f)$ of type $\l$ with
associated Weil representation $W:\Sp_{sm}(R)\to\GL(X)$.
Clearly $S=T\circ\beta:H(h)\to\GL(X)$ is an irreducible representation of $H(h)$ lying over $\mu$
and we see, as in the proof of Theorem \ref{x1},
that the restriction of $W$ to $\U_m(A)$ is a Weil representation.\qed

We proceed to examine various cases where (C1) and (C2) hold. Thus
$\U_m(A)\subseteq \Sp_{sm}(R)$ in each example, provided a
non-degenerate form $h:V\times V\to A$ of the specified type is
taken. This is always possible when $\varepsilon=1$, as well as
when $\varepsilon=-1$ and $m$ is even. Such form does not exist
when $\varepsilon=-1$ and $s,m$ are odd.

\begin{exa}
\label{e1} {\rm $A=M_n(R)$, with transposition as an involution,
$\varepsilon=-1$ and $d:A\to R$ the trace map.}
\end{exa}

\begin{exa}
\label{e2} {\rm $A=H(a,b)$, the generalized quaternion algebra
over $R$ associated to $a,b\in R^{\times}$, with the involution that
sends $i,j,k$ to their opposites, and $\varepsilon=-1$. Let
$d:A\to R$ pick the coefficient of 1.}
\end{exa}

\begin{exa}
\label{e3} {\rm We introduce the notion of quadratic extension of
$R$, ramified or unramified. Let $x\in R$. For an unramified
extension we take $x$ to be a unit that is not a square. For a
ramified extension we take $x$ to be a generator of the maximal
ideal of $R$. We consider the ring $A=R[y]$, where $y^2=x$ (thus
$A=R[t]/(t^2-x)$, a quotient of the polynomial ring $R[t]$). The
elements of $A$ are of the form $r+sy$ for unique $r,s\in R$, and
multiply in an obvious manner using $y^2=x$. It is easy to verify
that $A$ is a finite, local, principal ring of characteristic $p$.
In the unramified case the maximal ideal of $A$ is $\m + \m y$, where
$\m$ is the maximal ideal of $R$, and the
residue field is $F_{q^2}$. In the ramified case the maximal ideal of
$A$ is $Ay$ and the residue field is $F_q$. We have an involution
$*$ on $A$ defined by
$$
(r+sy)^*=r-sy,\quad r,s\in R.
$$

In the unramified case we let $\varepsilon=-1$ and take $d$ to be
the trace map $A\to R$, given by $a\mapsto a+a^*$. Explicitly,
$d(r+sy)=2r$.

In the ramified case we let $\varepsilon=1$
and let $d:A\to R$ be the defined by $d(r+sy)=2s$.
}
\end{exa}

\begin{exa} {\rm $R=F_q$
and $A=F_q[t]/(t^{s})$. The ring $A$ has an involution sending $t$
to $-t$, where $t^s=0$. We let $\varepsilon=-1$ if $s$ is odd,
and $\varepsilon=1$ if $s$ is even. Let $d:V\times V\to F_q$ be
given by
\begin{equation}
\label{lgf} d(a_0+a_1t+\cdots+a_{s-1}t^{s-1})=a_{s-1},\quad
a_i\in F_q.
\end{equation}}
\end{exa}

\begin{exa} {\rm $R=F_q$ and $A=F_q[G]$, the group algebra of a finite group $G$
of order $s$ over~$F_q$. Consider the involution $*$ on $A$
sending every $x$ in $G$ to $x^{-1}$. In this case
$\varepsilon=-1$. Let $d:A\to R$ be the $R$-linear map sending
each $a\in A$ to the coefficient of $1_G$.}
\end{exa}

\begin{exa}
\label{cli} {\rm This is a family of examples, including Examples
\ref{e1}-\ref{e3} above and many more.

Let $M$ be a free $R$-module of finite rank $t>0$ and let
$q:M\times M\to R$ be a symmetric bilinear form, not necessarily
non-degenerate. Let $A=C(M,q)$ be the associated Clifford algebra.
This is an associative unital algebra, which is free of rank $2^t$
as $R$-module. If $u_1,\dots,u_t$ form a basis of $M$ then a basis
for $A$ is formed by all
\begin{equation}
\label{base} u_1^{a_1}\cdots u_t^{a_t},\quad a_i=0,1,
\end{equation}
where $u_iu_j+u_ju_i=2q(u_i,u_j)$ for all $1\leq i,j\leq t$. We
may obtain $A$ as the quotient $T(M)/I$, where $T(M)$ is the
tensor algebra of $M$ and $I$ is the ideal generated by all
$u\otimes u-q(u,u)\cdot 1$, $u\in M$. It is not difficult to see
that $M$ admits an orthogonal basis relative to $q$, and we will
let $u_1,\dots,u_t$ denote such a basis.

(a) Here we take $\varepsilon=-1$. Let $d:A\to R$ be the map that
picks the coefficient of 1 relative to the basis (\ref{base}) of
$A$. Then (C1) holds if and only if $q$ is non-degenerate, i.e.
$q(u_i,u_i)\in R^{\times}$, $1\leq i\leq t$. To define an involution, let
$A^0$ be the opposite algebra of $A$. We have linear map $M\to
A^0$ given by $u\mapsto -u$. This can be extended to an algebra
anti-homomorphism $T(M)\to A$, given by $w_1\otimes\cdots\otimes
w_r\mapsto (-1)^r w_r\cdots w_1$, for $w_i\in M$. As the
generators $u\otimes u-q(u,u)\cdot 1$, $u\in M$, of $I$ are sent
to 0, we obtain an involution $A\to A$ such that $u\mapsto -u$ for
$u\in M$. We see that condition (C2) holds.

As an illustration, when $q(u_i,u_i)=-1$ and $t=2$ we obtain
Example \ref{e2}, and when $q(u_1,u_1)=x$ and $t=1$ we obtain the
unramified case Example \ref{e3}.

(b) Here we take $\varepsilon=-1$ if $t\equiv 0,3\mod 4$ and
$\varepsilon=1$ if $t\equiv 1,2\mod 4$. Let $d:A\to R$ be the map
that picks the coefficient of $u_1\cdots u_t$ relative to the
basis (\ref{base}) of $A$. Use the same involution as in case (a).
Then conditions (C1) and (C2) hold, regardless of the nature of
$q$ (this form can even be 0, in which case $A$ is the exterior
algebra of $M$). When $t=1$ and $q(u_1,u_1)=x$ we obtain the
ramified case of Example \ref{e3}.

(c) Here $\varepsilon=-1$, $t=2$, $q(u_i,u_i)=0$ and
$q(u_1,u_2)=1/2$. Then $u_1^2=0=u_2^2$, $u_1u_2+u_2u_1=1$, where
$u_1,u_2,u_1u_2,u_2u_1$ is a basis of $A$. We may construct, as in
case (a), an involution of $A$ that interchanges $u_1$ and $u_2$.
We let $d:A\to R$ pick the sum of the coefficients of $u_1u_2$ and
$u_2u_1$. This is Example \ref{e1} when $n=2$, with
$u_1,u_2,u_1u_2,u_2u_1$ playing the roles of the basic matrices
$e_{12},e_{21}, e_{11},e_{22}$.}
\end{exa}

\section{The Weil module in the case of a ramified quadratic extension}
\label{con}

We wish to study the restriction to the unitary group of the Weil
representation of the symplectic group. In view of Theorem
\ref{resu} we may consider the Weil representation of the unitary
group directly, as defined in \S \ref{HI}, and we shall do so.

The goal of this section is to construct a concrete Weil module
for the unitary group that can be advantageous when attempting
to decompose it into irreducible constituents.

For the remainder of the paper we will work exclusively within the
framework of the ramified case of Example \ref{e3}. Thus, $R$ is
as defined after Proposition \ref{x1} and
 $A=R[y]=R\oplus Ry$, where $y^2=x$. Here the maximal ideals of $R$
and $A$ are $\m=Rx$ and $\r=Ay$, with nilpotency degrees $\ell$
and $2\ell$, respectively, and residue fields $R/\m\cong F_q\cong
A/\r$. We have an involution $*$ on $A$ defined by
$(r+sy)^*=r-sy$. All ideals of $A$ are powers of $\r$ and hence
$*$-invariant. In addition, $V$ is a free right $A$-module of
finite rank~$m\geq 1$. Since $A$ is commutative, we may view $V$
as left module in an obvious way and we shall do so. Moreover,
$h:V\times V\to A$ is a non-degenerate $*$-hermitian form, with
unitary group $U$ and associated non-degenerate alternating
$R$-bilinear form $f:V\times V\to R$, given by $f(u,v)=d(h(u,v))$,
where $d(r+sy)=2s$. Note that $d(R)=0$.

We also fix from now on a primitive linear character
$\l:R^+\to\C^{\times}$, the primitive linear character $\mu:A^+\to\C^{\times}$,
given by $\mu(a)=\l(d(a))$, and the ideal $\i=\r^\ell$ of $A$.
Note that the annihilator, say $\i^\perp$, of $\i$ in $A$ is $\i$
itself.

\begin{lem}
\label{qz} There is an orthogonal basis $v_1,\dots,v_m$ of $V$
satisfying ~$h(v_i,v_i)\in R^{\times}$.
\end{lem}

\noindent{\it Proof.} This is \cite{15}, Lemma 2.2.  \qed

Consider the $U$-invariant submodule $V_0=\i V$ of $V$ and let
$$
V_0^\perp=\{v\in V\,|\,h(v,V_0)=0\}.
$$
It follows from Lemma \ref{qz} that $(\i V)^\perp=\i^\perp V$,
i.e., $V_0^\perp=V_0$.

We proceed to construct a specific Weil module for the unitary
group $U$, as guaranteed in \S \ref{HI}.

Let $H=H(h)$ be the Heisenberg group associated to $h$ and
consider the subgroup $(A,V_0)$ of $H$. We extend $\mu$ to a
linear character $\rho:(A,V_0)\to\C^{\times}$ by $\rho(a,u)=\mu(a)$. This
works because $h(V_0,V_0)=0$.

We claim that the stabilizer of $\rho$ in
$H$ is $(A,V_0)$. Indeed, for $(b,v)$ in $H$ we have
$$
\rho^{(b,v)}(a,u)=\rho(a,u)\mu(h(v,u)-h(u,v)).
$$
Suppose $(b,v)\in H$ stabilizes $\rho$. Then
$\mu(h(v,u)-h(u,v))=1$ for all $u\in V_0$. From $d(R)=0$ it
follows that $\mu(h(v,u)+h(u,v))=1$ for all $u\in V$. Therefore
$\mu(h(u,v))=1$ for all $u\in V_0$. The primitivity of $\mu$ now
implies $v\in V_0^\perp=V_0$, so $(b,v)\in (A,V_0)$, as claimed.

Let $Z=\C \zeta$ be a one-dimensional complex vector space. We have a representation of
$(A,V_0)$ on $Z$ given by
$$(a,u)\zeta=\rho(a,u)\zeta=\mu(a)\zeta,\quad a\in A,u\in V_0.$$

By Clifford theory the induced module $\mathrm{ind}_{(A,V_0)}^H Z$
is is irreducible and lies over $\mu$. Using the $U$-invariance of
$V_0$ we may obtain the same module in a more convenient form, as
follows. Extend the above action of $(A,V_0)$ on $Z$ to
$(A,V_0)\rtimes U$ by means of $(a,u)g \zeta=\mu(a)\zeta$, and consider
the induced module $X=\mathrm{ind}_{(A,V_0)\rtimes U}^{H\rtimes U}
Z$.  Let $T$ be a transversal for $V_0$ in $V$. Then $(0,v)$,
$v\in T$, is a transversal for both $(A,V_0)$ in $H$ and
$(A,V_0)\rtimes U$ in $H\rtimes U$. It follows that the
restriction of $X$ to $H$ is isomorphic to
$\mathrm{ind}_{(A,V_0)}^H Z$. Therefore, by construction, $res_H
X$ is Schr$\mathrm{\ddot{o}}$dinger of type $\mu$ and $res_U X$ is
Weil of type $\mu$. Now
$$
X=\C(H\rtimes U)\otimes_{\C((A,V_0)\rtimes U)} Z
$$
has $\C$-basis $e_v=(0,v)\otimes \zeta$, $v\in T$. The action of
$U$ on this basis is as follows:
\begin{equation}
\label{actu} ge_v=\mu(h(gv,v'))e_{v'},\quad g\in U,\; v,v'\in T,\;
gv\equiv v'\mod V_0.
\end{equation}
Indeed, we have
$$
\begin{aligned}
ge_v &=g(0,v)\otimes\zeta= g(0,v)g^{-1}g\otimes
\zeta=(0,gv)\otimes
g\zeta=(0,v'+(gv-v')) \otimes \zeta\\
&=(0,v')(0,gv-v')(-h(v',gv-v'),0)\otimes\zeta=\mu(-h(v',gv-v'))e_{v'}\\
&=\lambda(-f(v',gv-v'))e_{v'}=\lambda(f(gv-v',v'))e_{v'}=\mu(h(gv,v'))e_{v'}.
\end{aligned}
$$

For $v\in V$ let
$$
C(v)=\{g\in U\,|\, gv\equiv v\mod V_0\}.
$$
Thus, for $v\in T$, $C(v)$ is the stabilizer of the
subspace $\C e_v$ under the action of $U$.

It follows from (\ref{actu}) that for any if $v\in V$ then the map
$$\beta_v:C(v)\to\C^{\times}, \text{ given by }g\mapsto \mu(h(gv,v)), $$
is a group homomorphism. This can be verified independently of
(\ref{actu}).

Consider the subgroup $U(\i)$ of $U$ defined as follows:
$$
U(\i)=\underset{v\in V}\cap C(v)=\{g\in U\,|\, gv\equiv
v\mod V_0\text{ for all }v\in V\}.
$$
It is clear that $U(\i)$ is a normal subgroup of $U$. Using
$\i^2=0$ we see that $U(\i)$ is abelian.

Let $N=\{z\in A^{\times}\,|\, zz^*=1\}$ be the norm-1 subgroup of $A$. We identify $z\in N$ with the element of $Z(U)$ such that $v\mapsto zv$ for all $v \in V$.

For $v\in V$ we consider the subgroup $B(v)$ of $U$ given by
$$
B(v)=C(v)N.
$$
It is clear that $C(v)\subseteq B(v)$ with $C(v)$ normal in $B(v)$ and $B(v)/C(v)$ abelian.

For $v\in V$ let $\alpha_v:U(\i)\to\C^{\times}$ be the function defined
by
$$
\alpha_v(g)=\mu(h(gv,v)),\quad g\in U(\i).
$$
It follows from (\ref{actu}) that $\alpha_v$ is a group homomorphism, namely the restriction of $\beta_v$ to $U(\i)$.

\section{The stabilizer in $U$ of the linear character $\alpha_v$ of $U(\i)$}
\label{secsta}

A vector $v\in V$ is said to be primitive if $v\notin\r V$. This
is equivalent to say that $v$ belongs to a basis of $V$.

The goal of this section is to prove $\mathrm{Stab}_U
(\alpha_v)=B(v)$ when $v\in V$ is primitive. This will require a
few subsidiary results.

\begin{lem}
\label{genex} Let $a\in\i$ and $z_1,z_2\in V$. Consider the map $\rho_{a,z_1,z_2}:V\to V$ given by
$$
\rho_{a,z_1,z_2}(v)=v+ah(z_1,v)z_2-a^*h(z_2,v)z_1,\quad v\in V.
$$
Then $\rho_{a,z_1,z_2}\in U(\i)$.
\end{lem}

\noindent{\it Proof.} A direct calculation shows that $\rho_{a,z_1,z_2}\in U$. Since $a\in\i$, it follows that $\rho_{a,z_1,z_2}\in U(\i)$.\qed

\begin{lem}
\label{ext} Let $u_1,\dots,u_s$ be linearly independent vectors of $V$. Then $s\leq m$. Moreover, if $s=m$ then $u_1,\dots,u_s$ is a basis of $V$,
and if $s<m$ then $u_1,\dots,u_s$ can be extended to a basis of~$V$.
\end{lem}

\noindent{\it Proof.} Let $v_1,\dots,v_m$ be a basis of $V$. Write
$u_1$ in terms of this basis. Since $u_1$ is primitive, at least
one coefficient must be a unit. Thus, we can replace some basis vector, say
$v_1$, by $u_1$ and still have a basis. Repeat the process with
$u_2$. Since $u_1,u_2$ are linearly independent and $\r$ is nilpotent, some coefficient
other than the one from $u_1$ must be a unit, and we get a basis
starting with $u_1,u_2$. Repeating the above process we obtain the desired result.\qed

\begin{lem}
\label{dual} Let $v_1,\dots,v_s$ be linearly independent vectors of $V$ and let $a_1,\dots,a_s\in A$. Then there is a vector $v\in V$ such
that $h(v,v_i)=a_i$ for all $1\leq i\leq s$.
\end{lem}

\noindent{\it Proof.} By Lemma \ref{ext} we may extend the given
list to a basis $v_1,\dots,v_s,\dots,v_m$ of $V$. Consider the
linear functional $\phi:V\to A$ defined by $\phi(v_i)=a_i$ if
$i\leq s$ and $\phi(v_i)=0$ if $s<i\leq m$.

As the map $V\to V^*$ associated to $h$ is an isomorphism there is
a $v\in V$ such that $h(v,-)=\phi$.\qed

\begin{prop}
\label{hom6}
 Suppose $v,w\in V$ are primitive and satisfy
$\alpha_v=\alpha_w$. Then there exist $z\in N$ and a vector
$u_0\in\i V$ such that $w=zv+u_0$.
\end{prop}

\noindent{\it Proof.} By assumption
$$
\mu(h(gv,v))=\mu(h(gw,w)),\quad g\in U(\i).
$$

For ease of notation we will write $(u,v)$ instead of $h(u,v)$ for
the remainder of the proof.

For $a\in\i$ and $z_1,z_2\in V$ consider the element
$\rho_{a,z_1,z_2}$ of $U(\i)$ defined in Lemma \ref{genex}. Since
$\mu(R)=\l(d(R))=1$, we have
$$
\mu(v,v)=1.
$$
It follows that
$$
\mu(\rho_{a,z_1,z_2}(v),v)=\mu(v+a(z_1,v)z_2-a^*(z_2,v)z_1,v)=\mu(b-b^*),
$$
where
$$
b=a^*(v,z_1)(z_2,v).
$$
Likewise,
$$
\mu(\rho_{a,z_1,z_2}(w),w)=\mu(c-c^*),
$$
where
$$
c=a^*(w,z_1)(z_2,w).
$$
Again, since $\mu(R)=1$, we have
$$
 \mu(b+b^*)=1=\mu(c+c^*).
$$
It follows that
$$
\mu(2b)=\mu(2c),
$$
so
$$
\mu(a^*(v,z_1)(z_2,v))=\mu(a^*(w,z_1)(z_2,w)),\quad a\in\i,z_1,z_2\in V.
$$
The primitivity of $\mu$ and the fact that $\i$ is its own
annihilator imply
\begin{equation}
\label{run1} (v,z_1)(z_2,v)\equiv (w,z_1)(z_2,w)\mod\i,\quad z_1,z_2\in V.
\end{equation}
Suppose, if possible, that $v,w$ are linearly independent. Then
by Lemma \ref{dual} there is $z_1\in V$ such that
$(z_1,v)=1=(z_1,w)$. Substituting these values in (\ref{run1})
yields $$ (z_2,v)\equiv (z_2,w)\mod\i,\quad z_2\in V,$$ which means
$$ (z_2,v-w)\in\i,\quad z_2\in V.$$
We infer from Lemma \ref{qz} that $v-w\in\i V$, so $y^\ell
v-y^\ell w=0$, against the linear independence of $v,w$.

We deduce the existence of $a,b\in A$, not both zero, such that
$av=bw$. Since $v,w$ are primitive, we see that neither $a$ nor $b$ is
zero. Again, by the primitivity of $v,w$, it follows that $a=y^js$
and $b=y^jr$ for some units $r,s\in A$ and $0\leq j<2\ell$. Thus
$$
y^j(rw-sv)=0.
$$
Multiplying this by $r^{-1}$ we infer the existence of a unit $t\in A$ such that
$$
y^j(w-tv)=0
$$
The vectors annihilated by $y^j$ are those in $y^{2\ell-j}V$.
Letting $i=2\ell-j$, we see that $0<i\leq 2\ell$ and there is
$u\in V$ such that
\begin{equation}
\label{run2} w=tv+y^i u.
\end{equation}
Of all such expressions for $w$ in terms of $v$, with $t\in A^{\times}$,
$0<i\leq 2\ell$ and $u\in V$ we choose one with $i$ as large as
possible.

Suppose first $i=2\ell$. Then $w=tv$. Substituting this in (\ref{run1}) gives
$$
(tt^*-1)(v,z_1)(z_2,v)\in\i,\quad z_1,z_2\in V.
$$
By Lemma \ref{dual} there is $z_1=z_2\in V$ such that $(z_1,v)=1=(z_2,v)$, whence $tt^*\equiv 1\mod \i$.

Suppose next $i<2\ell$. Then, by the choice of $i$, $u$ must be
primitive. We claim that $v,u$ are linearly independent.
Otherwise, arguing as above we may find $a\in A^{\times}$, $0<k\leq
2\ell$ and $z\in V$ such that
$$
u=av+y^k z.
$$
But then
$$
w=tv+y^i(av+ y^k z)=(t+y^i a)v+y^{i+k}z,
$$
contradicting the choice of $i$. Thus $v,u$ are linearly
independent.

Substituting (\ref{run2}) in (\ref{run1}) yields
\begin{equation}
\label{run3} (tt^*-1)(v,z_1)(z_2,v)+b+c\in\i,
\end{equation}
where
$$
b=y^it^*(v,z_1)(z_2,u)+(-1)^i y^i t (u,z_1)(z_2,v)
$$
and
$$
c=(-1)^i y^{2i}(u,z_1)(z_2,u).
$$
Now there exists a unique $k$ such that $0\leq k\leq 2\ell$ and
$tt^*-1=y^k e$ for some unit $e\in A$. Three cases arise,
depending on how $k$ compares to $i$.

\noindent$\bullet$ $k<i$. By Lemma \ref{dual} we may choose
$z_1=z_2\in V$ such that $(z_1,v)=1=(z_2,v)$. Then (\ref{run3})
implies $\ell\leq k<i$, as required.

\noindent$\bullet$ $k>i$. By Lemma \ref{dual} we may choose
$z_1,z_2\in V$ such that
$$
t^*(v,z_1)=1, t(u,z_1)=(-1)^i, (z_2,u)=1, (z_2,v)=1.
$$
Then (\ref{run3}) implies $k>i\geq \ell$, as required.

\noindent$\bullet$ $k=i$. By Lemma \ref{dual} we may choose
$z_1,z_2\in V$ such that
$$
t^* (v,z_1)=1, t (u,z_1)=(-1)^{i+1}, (z_2,u)=1, (z_2,v)=1.
$$
Then (\ref{run3}) implies $k=i\geq \ell$.

We have shown that $w=tv+u_1$, where $tt^*\equiv 1\mod \i$ and
$u_1=y^iu\in \i V$. Now $tt^*=1+s$, where $s\in\i\cap R$. Since
$|R|$ has odd order the squaring map of the group $1+(\i\cap R)$
is surjective. Thus, there is $r\in \i\cap R$ such that
$$(1+r)(1+r)^*=(1+r)^2=1+s=tt^*.$$ Thus $t(1+r)^{-1}$ has norm 1,
so $t=(1+r)z$, where $z\in N$. Therefore,
$$w=(1+r)zv+u_1=zv+rzv+u_1=zv+u_0,
$$
where $u_0=rzv+u_1\in\i V$.\qed

\begin{thm}
\label{est}
 Let $v\in V$ be primitive vector. Then the stabilizer in $U$ of the linear character $\alpha_v:U(\i)\to\C^{\times}$
is  $B(v)$.
\end{thm}

\noindent{\it Proof.} As $B(v)=C(v)N$, the group $N$ central in
$U$ and $\alpha_v$ is the restriction of $\beta_v:C(v)\to\C^{\times}$ to
$U(\i)$,
 we see that $B(v)$ stabilizes $\alpha_v$. On the other hand, suppose $g_0\in
U$ stabilizes $\alpha_v$ and $w=g_0v$. Then $w$ is primitive and
every $g\in U(\i)$ satisfies
$$
\alpha_v(g)=\alpha_v(g_0^{-1}gg_0)=\mu(h(g_0^{-1}gg_0v,v))=\mu(h(gg_0v,g_0v))=\mu(h(gw,w))=\alpha_w(g).
$$
It follows from Proposition \ref{hom6} that $g_0\in B(v)$.\qed

\section{Irreducible constituents of $Top$}
\label{sectop}

We wish to decompose the Weil module $X$ into irreducible
constituents. Let $Top$ be the subspace of $X$ spanned by all
$e_v$ with $v\in T$ primitive, and $Bot$ the subspace of $X$
spanned by all $e_v$ with $v\in T$ not primitive. Then
$$
\dim\,Bot=q^{(\ell-1)m},\; \dim\,Top=q^{\ell m}-q^{(\ell-1)m}.
$$
In view of the action (\ref{actu}) of $U$ on $X$, it is clear that
$Top$ and $Bot$ are $U$-submodules of $X$. We clearly have
$X=Top\oplus Bot$. We will see in \S\ref{secbot} that if $\ell=1$
then $Bot$ is the trivial $U$-module and if $\ell>1$ then $Bot$
affords a Weil representation of primitive type for a unitary
group of rank $m$ associated to a ramified quadratic extension,
where the nilpotency degree of the maximal ideal of the base ring
is $\ell-1$. Thus, reasoning by induction on $\ell$, it suffices
to decompose $Top$ into irreducible constituents. This is the
purpose of this section.

Let $\chi$ be an irreducible constituent of $Top$. By (\ref{actu})
the restriction of $Top$ to the abelian normal subgroup $U(\i)$ is
the sum of all $\alpha_v$, $v\in T$ with $v$ primitive.
Thus $\chi$ lies over a given $\alpha_v$, with $v$ primitive.
 By Theorem \ref{est} the stabilizer of $\alpha_v$ is $B(v)$. Now
 $B(v)=NC(v)$ and
$\alpha_v$ extends to the linear character $\beta_v:C(v)\to\C^{\times}$.
Since $N$ is abelian, we may extend $\beta_v$ to a linear
character $\gamma_v:B(v)\to\C^{\times}$ (see Lemma \ref{abex}). By
Gallagher's theorem (Corollary 6.17 of \cite{10}) the irreducible
characters of $B(v)$ lying over $\alpha_v$ are of the form
$\gamma_v\tau$, where $\tau$ runs over all irreducible characters
of $B(v)$ that are trivial on $U(\i)$. By Clifford Theory (Theorem
6.11 of \cite{10}) there is bijection $\tau\mapsto
ind_{B(v)}^{U}{\gamma_v\tau}$ from the set of all irreducible
characters of $B(v)$ lying over $\alpha_v$ and those of $U$ lying
over $\alpha_v$. Thus $\chi=ind_{B(v)}^{U}{\gamma_v\tau}$ for a
unique $\tau$. We will show that the $\tau$ that actually occur in
the decomposition of $Top$ are those trivial on $C(v)$. Since
$B(v)/C(v)$ is abelian, all these $\tau$ are actually linear
characters, whence $\chi$ is a monomial character.

Consider the equivalence relation $\sim$ on $V\setminus yV$, given
by
$$
v\sim w\text{ if there is }g\in U\text{ such that }gv\equiv w\mod
\i V.
$$
Let $S$ be set of representatives of primitive vectors for $\sim$.
We may assume that $S$ is contained in $T$, the transversal for
$\i V$ in $V$ used in \S\ref{con}. For $s\in S$ let
$$
Top(s) = \mbox{$U$-submodule of $Top$ generated by $e_s$.}
$$
By means of
(\ref{actu}) we see that $Top(s)$ has a basis consisting of all
$e_v$ such that $v$ is primitive and $v\sim s$.

\begin{prop}
\label{sor}
The $U$-submodules $Top(s)$, $s\in S$, are disjoint, i.e., they have no isomorphic irreducible constituents in common.
Moreover, their sum is $Top$. In particular,
$$
Top=\underset{s\in S}\oplus  Top(s).
$$
\end{prop}

\noindent{\it Proof.} Let $v\in T$ be primitive. Then $v\sim s$ for some $s\in S$. Thus $gs\equiv v\mod\i V$ for some $g\in U$
and by (\ref{actu}) $e_v\in Top(s)$. This shows that $Top$ is contained in the given sum. The linear characters
of $U(\i)$ appearing in $Top(s)$ are all $\alpha_w$, where $w\in T$ is primitive and $w\sim s$. By Proposition \ref{hom6}
the set of these $\alpha_w$ is disjoint from the set corresponding to any other $s'\in S$, $s'\neq s$.\qed

Let $v\in V$. We have the linear character $\beta_v\vert_{C(v)\cap N}$ of $C(v)\cap N$. This is an abelian group,
so we may extend $\beta_v\vert_{C(v)\cap N}$ to a linear character $\delta_v:N\to\C^{\times}$. We define
$$
\gamma_v(gz)=\beta_v(g)\delta_v(z),\quad g\in C(v), z\in N.
$$
\begin{lem}
\label{abex} $\gamma_v$ is a well-defined linear character of $B(v)$.
\end{lem}

\noindent{\it Proof.} Suppose $g_1z_1=g_2z_2$ for $g_1,g_2\in C(v)$ and $z_1,z_2\in N$. Then
$g_2^{-1}g_1=z_2z_1^{-1}\in C(v)\cap N$. Since $\delta_v$ extends $\beta_v\vert_{C(v)\cap N}$, we have
$$
\delta_v(z_2z_1^{-1})=\beta_v(z_2z_1^{-1})=\beta_v(g_2^{-1}g_1).
$$
Thus
$$
\delta_v(z_2)\delta(z_1^{-1})=\beta_v(g_2^{-1})\beta_v(g_1),
$$
so
$$
\beta_v(g_2)\delta_v(z_2)=\beta_v(g_1)\delta(z_1).\qed
$$

Suppose next $v\in V$ is primitive. We easily see that in this
case
$$
C(v)\cap N=(1+\i)\cap N,
$$
so
\begin{equation}
\label{indbv} B(v)/C(v)=C(v)N/C(v)\cong N/C(v)\cap N=N/(1+\i)\cap
N \cong (1+\i)N/(1+\i).
\end{equation}

Thanks to (\ref{indbv}) we may identify $B(v)/C(v)$ with $N/N\cap
(1+\i)$. It follows from Lemma \ref{abex} that the linear
characters of $B(v)$ extending $\beta_v$ are the form
$\phi\gamma_v$, where $\phi$ is a linear character of $N$ trivial
on $N\cap (1+\i)$, considered as a linear character of $B(v)$
trivial on $C(v)$.

Let $G$ be the group of linear characters of $N$ trivial on $N\cap
(1+\i)$. For $v\in T$, still primitive, and $\phi\in G$ we
consider the element $E_{\phi,v}$ of $Top$ defined as follows:
$$
E_{\phi,v}=(\underset{z\in N}\sum \phi(z^{-1})\gamma_v(z^{-1})z)e_v.
$$
Then $zE_{\phi,v}=\phi(z)\gamma_v(z)E_{\phi,v}$ and $gE_{\phi,v}=\beta_v(g)E_{\phi,v}$ for $z\in N$ and $g\in C(v)$.

\begin{lem}
\label{zac} Let $v\in T$ be primitive. Then

(a) $E_{\phi,v}\neq 0$.

(b) The stabilizer of $\C E_{\phi,v}$ in $U$ is $B(v)$.
\end{lem}

\noindent{\it Proof.} (a) Let $D$ be a transversal for $M=N\cap (1+\i)$ in
$N$. Then by (\ref{actu})
{\small
\begin{equation}
\label{ela}
E_{\phi,v}=(\underset{b\in D}\sum\phi(b^{-1})\gamma_v(b^{-1})b)(\underset{z\in
M}\sum \beta_v(z^{-1})z)e_v=(\underset{b\in
D}\sum\phi(b^{-1})\gamma_v(b^{-1})d)|M|e_v\neq 0.
\end{equation}
}
(b) This follows from (\ref{actu}).\qed

\begin{thm}
\label{dsm} Fix $s\in S$. For $\phi\in G$ let $Top(\phi,s)$ be the $U$-submodule of $Top(s)$ generated by $E_{\phi,s}$. Then

(a) The character of $Top(\phi,s)$ is $ind_{B(v)}^U \phi\gamma_v$.

(b) The $U$-module $Top(\phi,s)$ is irreducible.

(c) The $U$-modules $Top(\phi,s)$, $\phi\in G$, are pairwise
non-isomorphic.

(d) $Top(s)=\underset{\phi\in G}\oplus Top(\phi,s)$.
\end{thm}

\noindent{\it Proof.} (a) By Lemma \ref{zac} $\C E_{\phi,s}$ is 1-dimensional with stabilizer $B(v)$, which acts on
$\C E_{\phi,s}$ via $\phi\gamma_v$. It follows from (\ref{actu}) that $Top(\phi,s)$ has character $ind_{B(v)}^U \phi\gamma_v$.

(b),(c) This follows from (a), Theorem \ref{est} and Clifford Theory.

(d) As the $E_{\phi,s}$, $\phi\in G$, are linearly independent, it
follows from (\ref{ela}) that $e_s$ is in their span.\qed

\begin{thm}
\label{m3} (a) $Top$ is multiplicity free, with the following monomial irreducible constituents:
$$
Top=\underset{s\in S}\oplus\underset{\phi\in G}\oplus Top(\phi,s).
$$
(b) Let $\iota$ be the central involution of $U$ defined by
$v\mapsto -v$, and let $Top^+$ and $Top^-$ be the eigenspaces of
$\iota$ acting on $Top$. Let $G^+$ be the subgroup of $G$ of all
$\phi$ such that $\phi(\iota)=1$ and let $G^-$ be the coset of
$G_1$ in $G$ of all $\phi$ such that $\phi(\iota)=-1$. Then
$Top^+$ and $Top^-$ are $U$-submodules of $Top$, with
$$Top=Top^+\oplus Top^-,\quad \dim\,Top^+=\dim\,Top^-=(q^{\ell m}-q^{(\ell-1)m})/2$$
and $Top^+$, $Top^-$ have the following decompositions in irreducible constituents:
$$
Top^+=\underset{s\in S}\oplus\underset{\phi\in G^+}\oplus
Top(\phi,s),\quad Top^-=\underset{s\in S}\oplus\underset{\phi\in
G^-}\oplus Top(\phi,s).
$$
\end{thm}

\noindent{\it Proof.} (a) This follows from Proposition \ref{sor} and Theorem \ref{dsm}.

(b) Only the statement about the dimensions of $Top^\pm$ requires a proof. To verify these dimensions, select the transversal $T$ for $\i V$ in $V$
to be symmetric,
in the sense that $v\in T$ if and only if $-v\in T$. Then $Top^+$ and $Top^-$ are
respectively spanned by $e_v+e_{-v}$ and $e_v-e_{-v}$ as $v$ runs through the primitive vectors of $T$.\qed

\section{Counting the irreducible constituents of $Top$}
\label{seccount}

This section finds the exact number of irreducible constituents of
$Top$. It turns out to be equal to the number of $U$-orbits of
$V\setminus y^2 V$. This fact is independently verified in
\S\ref{secbot}.

\begin{thm} Let $v,w\in V$ be primitive vectors satisfying $h(v,v)=h(w,w)$. Then there exists $g\in U$ such that $gv=w$.
\label{crux}
\end{thm}

\noindent{\it Proof.} This can be found in \cite{15}, Theorem 4.1. \qed

\begin{prop}
\label{tot2}
The number of $U$-orbits of vectors in  $y V\setminus y^2 V$ is equal to the number of $U$-orbits of vectors in $V\setminus yV$.
\end{prop}

\noindent{\it Proof.} Let $E$ be a set of representatives for the
$U$-orbits of $V\setminus yV$. It is clear that every vector in
$yV\setminus y^2V$ is $U$-conjugate to a vector in $yE$. Thus the
map $E\to yE$, given by $e\mapsto ye$, is surjective. We claim
that it is also injective and, in fact, that if $e_1,e_2\in E$ and
$ye_1,ye_2$ are $U$-conjugate then $e_1=e_2$. For this purpose we
view $yV$ as a module for $A/y^{2\ell-1}A$ and consider the map
$q:yV\times yV\to A/y^{2\ell-1}A$ given by
$$
q(yv,yw)=h(v,w)+y^{2\ell-1}A,\quad v,w\in V.
$$
Given $g\in U$, for all $v,w\in V$ we have
$$
q(gyv,gyw)=q(ygv,ygw)=h(gv,gw)+y^{2\ell-1}A=h(v,w)+y^{2\ell-1}A=q(yv,yw),
$$
so the restriction of $g$ to $yV$ preserves $q$. Suppose $g\in U$
satisfies $gye_1=ye_2$. By above,
$$q(ye_2,ye_2)=q(gye_1,gye_1)=q(ye_1,ye_1),$$
which means
$$
h(e_1,e_1)-h(e_2,e_2)\in y^{2\ell-1}A\cap R=x^{\ell-1}yA\cap R=0.
$$
Thus $e_1,e_2$ are $U$-conjugate by Theorem \ref{crux}. Since
$e_1,e_2\in E$, we infer $e_1=e_2$.\qed

\begin{prop}
\label{hurra} Consider the equivalence relation $\simeq$ on
$V\setminus yV$, given by
$$
v\simeq w\text{ if }h(v,v)\equiv h(w,w)\mod \i\cap R.
$$
Then $\simeq$ is equal to the equivalence relation $\sim$ defined
in \S\ref{sectop}.
\end{prop}

\noindent{\it Proof.} Suppose first that $v\sim w$. Then $gv=w+u$
for some $g\in U$ and $u\in \i V$. Therefore
$$
h(v,v)=h(gv,gv)=h(w+u,w+u)=h(w,w)+h(w,u)+h(w,u)^*\equiv
h(w,w)\!\!\!\mod \i\cap R.
$$
Therefore $v\simeq w$.

Suppose conversely that $v\simeq w$. Then $h(v,v)=h(w,w)+r$, where
$r\in\i\cap R$. Since $w$ is primitive there is a basis
$w,w_2,\dots,w_m$ of $V$. As $h$ is non-degenerate there is $z\in
V$ such that $h(w,z)=r/2$ and $h(w_i,z)=0$. In particular,
$h(z,V)\subset \i$, so $z\in \i V$ by Lemma \ref{qz}. Using
$h(z,z)=0$ and $h(w,z)=r/2$ we deduce
$$
h(w+z,w+z)=h(w,w)+r=h(v,v).
$$
By Theorem \ref{crux} there is $g\in V$ such that $gv=w+z$, so
$v\sim w$.\qed

\begin{lem}
\label{sqb}  $[R^{\times}:R^{\times 2}]=2$.
\end{lem}

\noindent{\it Proof.} The kernel of the canonical map $R^{\times}\to
(R/\m)^{\times}$ is $1+\m$, which is a finite group of odd order. It
follows that the squaring map $1+\m\to 1+\m$ is an epimorphism.
This implies that $r$ is a square in $R^{\times}$ if and only if $r+\m$
is a square in $(R/\m)^{\times}$. Since $R/\m\cong F_q$ and
$[F_q^{\times}:F_q^{\times 2}]=2$, the result follows.\qed

\begin{lem}
\label{qubo} Let $Q:A^{\times} \to R^{\times}$ be the norm map $a\mapsto aa^*$.
Then $Q(A^{\times})=R^{\times 2}$.
\end{lem}

\noindent{\it Proof.} Clearly $R^{\times 2}\subseteq Q(A^{\times})$. In view of
Lemma \ref{sqb} it suffices to show that $Q$ is not surjective.

The involution that $*$ induces on $A/\r$ is the identity map,
whence the induced norm map $(A/\r)^{\times} \to (R/\m)^{\times}$ is the squaring
map of $F_q^{\times}$, so $Q$ is not surjective.\qed

\begin{thm}
\label{xtipo} Up to equivalence, there are exactly two non-degenerate
hermitian forms on $V$, depending on whether the determinant of the form, relative to any basis, is a square or not.
\end{thm}

\noindent{\it Proof.} This follows from \cite{15}, Corollary 3.6. \qed

\medskip

Given $r_1,\dots,r_m\in R^{\times}$ we say that $h$ is of type $\{r_1,\dots,r_m\}$ if there is a basis $B$ of $V$ relative to which $h$ has
matrix $\mathrm{diag}\{r_1,\dots,r_m\}$. In that case, $h$ is also of type  $\{s_1,\dots,s_m\}$, for $s_i\in R^{\times}$, if and only if
$(r_1\cdots r_m)(s_1\cdots s_m)^{-1}\in R^{\times 2}$.

We fix an element $\epsilon$ in $R^{\times}$ but not in $R^{\times 2}$. Thus
$R^{\times}=R^{\times 2}\cup \epsilon  R^{\times 2}$.

Let $\Lambda$ be the set of all values $h(u,u)$ taken on primitive
vectors $u\in V$ and let $K$ be the number of $U$-orbits of
$V\setminus yV$. It follows from Theorem \ref{crux} that
$K=|\Lambda|$.

\begin{thm}
\label{tiso} (a) Suppose $m=1$. If $h$ is of type $\{1\}$ then
$\Lambda=R^{\times 2}$ and if $h$ is of type $\{\epsilon\}$ then
$\Lambda=R^{\times} \setminus R^{\times 2}$.

(b) Suppose $m=2$. If $h$ is of type $\{1,-1\}$ then $\Lambda=R$ and if
$h$ is of type $\{1,-\epsilon\}$ then $\Lambda=R^{\times}$.

(c) If $m>2$ then  $\Lambda=R$.
\end{thm}

\noindent{\it Proof.} This is \cite{15}, Lemma 3.7. \qed

\begin{thm}
\label{nexo}
Let $L$ be the number of irreducible constituents of $Top$. Then $L$ is equal to the number of $U$-orbits of $V\setminus y^2V$.
\end{thm}

\noindent{\it Proof.} Let $S$ be set of representatives of primitive vectors for the equivalence relation $\sim$ defined in \S\ref{sectop}.
We know from Proposition \ref{sor} that $Top$ is the direct sum of $|S|$ submodules $Top(s)$.
Furthermore, by Theorem \ref{dsm}, each $Top(s)$ is the direct sum of $|N/N\cap (1+\i)|=|N(1+\i)/(1+\i)|$ irreducible constituents. Thus
$$
L=|S||N(1+\i)/1+\i|
$$
On the other hand, by Proposition \ref{tot2}, the number of $U$-orbits of $V\setminus y^2V$ is $2K$, where $K$
is the number of $U$-orbits of $V\setminus yV$.
Thus, we need to verify that
\begin{equation}
\label{k2}
2K=|S||(1+\i)N/(1+\i)|.
\end{equation}
We claim that
\begin{equation}
\label{y2}
|S|=K/|R\cap\i|.
\end{equation}
Indeed, Proposition \ref{hurra} ensures that $|S|$ is the number
of values $h(u,u)$, with $u$ primitive, that are distinct modulo
$R\cap\i$. Moreover, as indicated above, $K=|\Lambda|$ is the
number values $h(u,u)$, with $u$ primitive. A full description of
$\Lambda$ is given in Theorem \ref{tiso}. Three cases arise.
$\Lambda=R$, in which case $|S|=|R|/|R\cap\i|=K/|R\cap\i|$;
$\Lambda=R^{\times 2}$ or $\Lambda=R^{\times}\setminus R^{\times 2}$, and in both
cases $|S|=|R^{\times 2}/1+R\cap\i|=K/|R\cap\i|$; $\Lambda=R^{\times}$, in
which case $|S|=|R^{\times}/1+R\cap\i|=K/|R\cap\i|$. This proves
(\ref{y2}) in all cases.

Substituting (\ref{y2}) in (\ref{k2}) reduces our verification to
\begin{equation}
\label{k3}
2|R\cap\i|=|(1+\i)N/(1+\i)|.
\end{equation}
Now $A/\i$ inherits an involution from $A$ and $R/R\cap\i\cong
(R+\i)/\i$ naturally imbeds in $A/\i$ as the fixed ring of this
involution, yielding a norm map $(A/\i)^{\times}\to (R/R\cap \i)^{\times}$. The
proofs of Lemmas \ref{sqb} and \ref{qubo} apply to show that the
image of the norm map $(A/\i)^{\times}\to (R/R\cap \i)^{\times}$ has index 2 in
$(R/R\cap \i)^{\times}$ and coincides with the group of squares of
$(R/R\cap \i)^{\times}$. The kernel of the norm map $(A/\i)^{\times}\to (R/R\cap
\i)^{\times}$  is easily seen to be $N(1+\i)/(1+\i)$. We deduce
$$
|N(1+\i)/(1+\i)|\times |(R/R\cap \i)^{\times 2}|=|(A/\i)^{\times}|,
$$
or
$$
|N(1+\i)/(1+\i)|\times |(R/R\cap \i)^{\times}|=2|(A/\i)^{\times}|.
$$
Since $R/R\cap\i$ and $A/\i$ are local rings, we see that $R^{\times}\to (R/R\cap \i)^{\times}$ and $A^{\times} \to (A/\i)^{\times}$
are epimorphisms with kernels $1+R\cap\i$ and $1+\i$, respectively. Therefore
\begin{equation}
\label{k4}
|N(1+\i)/(1+\i)|\times |R^{\times}|/|R\cap \i|=2|A^{\times}|/|\i|.
\end{equation}
On the other hand, from
$$
q^\ell-q^{\ell-1}=(q^{2\ell}-q^{2\ell-1})/q^\ell,
$$
we deduce
$$
|R^*|=|A^{\times}|/|\i|.
$$
Substituting this in (\ref{k4}) yields (\ref{k3}).

\begin{cor}
\label{uva} The number of irreducible constituents of $Top$ is:

\noindent $\bullet$ $q^\ell-q^{\ell-1}$ if $m=1$.

\noindent $\bullet$ $2(q^\ell-q^{\ell-1})$ if $m=2$ and $h$ is of type $\{1,-\epsilon\}$.

\noindent  $\bullet$ $2q^\ell$ if $m>2$ or $m=2$ and $h$ is of type $\{1,-1\}$.
\end{cor}

\section{Bottom is a Weil module for a unitary group over a quotient ring}
\label{secbot}

This section is similar in spirit to \S 5 and \S 6 of \cite{1}, with
some technical differences.

\begin{lem}
\label{b1}
 The subgroup $(0,y^{2\ell-1}V)$ of the Heisenberg
group $H$ acts trivially on $Bot$.
\end{lem}

\noindent{\it Proof.} Let $v\in yV\cap T$ and $w\in y^{2\ell-1}V$.
We wish to see that $(0,w)e_v=e_v$. Since $h(v,w)=0$ and
$y^{2\ell-1}V\subseteq y^\ell V$, we have
$$
(0,w)e_v=(0,w)(0,v)\otimes z=(0,v)(0,w)\otimes z=(0,v)\otimes
(0,w)z=(0,v)\otimes z=e_v.
$$

\begin{prop}
\label{b2} We have an isomorphism of $H$-modules
$$
X\cong ind_{(A,y V)}^H Bot,
$$
where $Bot$ is an irreducible $(y^2A, y V)$-module of dimension
$q^{(\ell-1)m}$.
\end{prop}

\noindent{\it Proof.} Let $X(y^{2\ell-1})$ be the fixed points of
$(0,y^{2\ell-1}V)$ in $X$. Consider the normal subgroup
$(A,y^{2\ell-1}V)$ of $H$ and its linear character
$\rho:(A,y^{2\ell-1}V)\to\C^{\times}$, given by $\rho(a,v)=\mu(a)$. Using
$d(R)=0$ and the primitivity of $\mu$ we see that the stabilizer of $\rho$ in $H$ is $(A,y V)$.
By definition
$$
X(y^{2\ell-1})=\{z\in X\,|\, (a,v)z=\rho(a,v)z=\mu(a)z\text{ for all }a\in A,
v\in y^{2\ell-1}V\}.
$$
By Lemma \ref{b1} we know that $Bot\subseteq X(y^{2\ell-1})$.
Conclusion: $\rho$ enters the restriction of $X$ to
$(A,y^{2\ell-1}V)$; the $\rho$-homogeneous component for
$(A,y^{2\ell-1}V)$ in $X$ is $X(y^{2\ell-1})$, which is an
irreducible $(A,y V)$-module, and $X\cong ind_{(A,y V)}^H
X(y^{2\ell-1})$. Since $(A,0)$ is central in $H$, it follows that
$X(y^{2\ell-1})$ is irreducible as a module for $(y^2A, y V)$.
Counting dimensions
$$
q^{\ell m}=\dim\, X=[H:(A,y V)]\times \dim\,
X(y^{2\ell-1})=q^m\times \dim\, X(y^{2\ell-1}),
$$
so
$$
\dim\, X(y^{2\ell-1})=q^{(\ell-1)m}=\dim\, Bot,
$$
whence
$$
X(y^{2\ell-1})=Bot.  \hfill \qed
$$

Suppose that $\ell>1$. Let $\overline{R}=R/x^{\ell-1} R$ and
$\overline{A}=A/y^{2(\ell-1)} A$. Then $\overline{A}/\overline{R}$
is a ramified quadratic extension and $\overline{A}$ inherits an
involution from $A$ whose fixed points form $\overline{R}$. Let
$\overline{V}=V/y^{2(\ell-1)}V$. Then $h$ gives rise to a
non-degenerate hermitian form $\overline{h}:\overline{V}\times
\overline{V}\to \overline{A}$, given by
$$
\overline{h}(v+y^{2(\ell-1)}V,
w+y^{2(\ell-1)}V)=h(v,w)+y^{2(\ell-1)}A,\quad v,w\in V.
$$
Let $\overline{U}$ and $\overline{H}$ be the associated unitary and Heisenberg groups.

Consider the map $(y^2A, y V)\to \overline{H}$ given by
\begin{equation}
\label{yus} (y^2 a,yv)\mapsto (-a+y^{2(\ell-1)}A,
v+y^{2(\ell-1)}V).
\end{equation}
It is a well-defined group epimorphism with kernel
$(0,y^{2\ell-1}V)$. The corresponding isomorphism $(y^2A, y V)/(0,y^{2\ell-1}V)\to \overline{H}$
is compatible with the actions of $U/U(y^{2(\ell-1)})$ on $(y^2A, y V)/(0,y^{2\ell-1}V)$ and $\overline{H}$.

As $(0,y^{2\ell-1}V)$ acts trivially on
$Bot$, we see that $Bot$ is an irreducible module for the quotient group $(y^2A,
y V)/(0,y^{2\ell-1}V)$, and hence for $\overline{H}$ via
(\ref{yus}). Let $\overline{\mu}:\overline{A}^+\to\C^{\times}$ be the
linear character defined by
$$
\overline{\mu}(a+y^{2(\ell-1)}A)=\mu(-y^2a),\quad a\in A.
$$
Then $\overline{\mu}$ is primitive. It follows from Proposition
\ref{b2} and Corollary \ref{corsch} that the representation of
$\overline{H}$ afforded by $Bot$ via (\ref{yus}) is
Schr$\mathrm{\ddot{o}}$dinger of type $\overline{\mu}$.

We now remove the condition that $\ell>1$. For $j\geq 0$ we
consider the normal subgroup $U(y^j)$ of $U$ defined as follows:
$$
U(y^j)=\{g\in U\,|\, gv\equiv v\mod y^j V\text{ for all }v\in V\}.
$$

\begin{thm}
\label{notrivial}
 (a) The congruence subgroup $U(y^{2(\ell-1)})$ acts
trivially on $Bot$.

(b) The congruence subgroup $U(y^{2(\ell-1)})$ does not act trivially
on any of the irreducible constituents of $Top$.
\end{thm}

\noindent{\it Proof.} (a) If $\ell=1$ then $Bot=\C e_0$ and
(\ref{actu}) shows that $U$ acts trivially on $Bot$. Suppose
$\ell>1$. Then $2(\ell-1)\geq \ell$, so $U(y^{2(\ell-1)})\subseteq
U(y^\ell)$. Let $yv\in T\cap y V$. Then $U(y^{2(\ell-1)})$ acts on
the basis vector $e_{yv}$ of $Bot$ via the linear character
$g\mapsto \mu(h(gyv,yv))$. By assumption
$$
gv=v+y^{2(\ell-1)}u\text{ for some }u\in V,
$$
so
$$
\mu(h(gyv,yv))=\mu(-xh(gv,v))=\mu(-xh(v,v))\mu(-xh(y^{2(\ell-1)}u,v))=1.
$$
(b) Suppose first that $\ell>1$. Then $U(y^{2(\ell-1)})\subseteq
U(y^\ell)$. The irreducible constituents of $Top$ restricted to
$U(y^{2(\ell-1)})$ are the linear characters $g\mapsto
\mu(h(gv,v))$ for $v$ primitive. Suppose one of these is trivial.
Reasoning as in the proof of Proposition \ref{hom6} (with $a\in
y^{2(\ell-1)}A$) we see that
$$
h(v,z_1)h(z_2,v)\in y^2A,\quad z_1,z_2\in V.
$$
Choosing $z_1=z_2$ so that $h(v,z_1)=1=h(z_2,v)$ we reach a
contradiction.

Suppose next $\ell=1$. Then $U(y^{2(\ell-1)})=U$, which does not
act trivially on any irreducible constituent of $Top$, as
$U(y^\ell)=U(y)$ itself does not. Indeed, it acts through a linear
character that is not trivial by the same argument just made
above.\qed

\begin{cor} $Bot$ coincides with the fixed points of
$U(y^{2(\ell-1)})$ in $X$.
\end{cor}

\begin{thm} The natural map $U\to \overline{U}$, given by
$g\mapsto \overline{g}$, where
$\overline{g}(v+y^{2(\ell-1)}V)=g(v)+y^{2(\ell-1)}V$, is a group
epimorphism with kernel $U(y^{2(\ell-1)})$.
\end{thm}

\noindent{\it Proof.} It is clear that this is a group
homomorphism with kernel $U(y^{2(\ell-1)})$. The fact that it is
surjective follows from \cite{15}, Theorem 5.2.  (This is also shown in \S 4 of
\cite{8}.) \qed

\begin{thm}
\label{cand} The representation of $\overline{U}$ on $Bot$ obtained
via the isomorphism $U/U(y^{2(\ell-1)})\to\overline{U}$ is a Weil
representation of primitive type $\overline{\mu}$.
\end{thm}

\noindent{\it Proof.} Let $S:H\to\GL(X)$ and $W:U\to\GL(X)$ be the Schr$\mathrm{\ddot{o}}$dinger and Weil representations of type $\mu$,
as constructed in \S\ref{con}. Then
$$
W(g)S(k)W(g)=S({}^g k),\quad g\in U,k\in H.
$$
First restrict $k$ to $(y^2A,yV)$ and all above operators to $Bot$. Then, using the fact that $U(y^{2(\ell-1)})$ and $(0,y^{2\ell-1})$
act trivially on $Bot$, pass to the corresponding representations of $\overline{U}$ and $\overline{H}$, to obtain the desired result.\qed

\section{Irreducible constituents of the Weil module $X$}
\label{secfull}

Let $\overline{R},\overline{A},\overline{V},\overline{h}$ and
$\overline{U}$ be defined as in \S\ref{secbot}.

\begin{note} The irreducible constituents of the Weil module $X$ are those
of $Top$, as described in Theorem \ref{m3}, plus those of $Bot$.
Moreover, either $\ell=1$, in which case $Bot$ is the trivial
$U$-module, or $\ell>1$ and $Bot$ is Weil module for
$\overline{U}$, in which case its irreducible constituents can be
recursively obtained by repeating this process.
\end{note}

\begin{thm}
\label{mulfree} The Weil module $X$ of $U$ is multiplicity free.
\end{thm}

\noindent{\it Proof.} Theorem \ref{notrivial} shows that $Bot$ and
$Top$ are disjoint, while $Top$ is multiplicity free by Theorem
\ref{m3}. As seen in \S\ref{secbot}, $Bot$ is the trivial
$U$-module when $\ell=1$ and affords a Weil representation for
$\overline{U}$ when $\ell>1$, so the result follows by
induction.\qed

\begin{thm}
\label{orbred} If $\ell>1$ the number of orbits of $\overline{U}$ acting on
$\overline{V}$ is equal to the number of orbits of $U$ acting on
$y^2V$.
\end{thm}

\noindent{\it Proof.} Since $U\to \overline{U}$ is an epimorphism the $U$ and
$\overline{U}$ orbits of $\overline{V}$ are the same. On the other
hand, the map $y^2V\to \overline{V}$ given by $y^2v\mapsto
v+y^{2(\ell-1)}V$ is a bijection compatible with the actions of
$U$, i.e., $y^2V$ and $\overline{V}$ are equivalent $U$-sets.
Thus, the number of $U$-orbits of $y^2V$ equals the number of
$U$-orbits of $\overline{V}$, which is the number of
$\overline{U}$-orbits of $\overline{V}$.\qed

\begin{note} It follows from Theorems \ref{comega}, \ref{notrivial}, \ref{cand}, and \ref{orbred} that
the number of irreducible constituents of $Top$ is equal to
the number of $U$-orbits of $V\setminus y^2 V$. This was independently verified in Theorem \ref{nexo}.
\end{note}

\begin{thm}
\label{totalorb} The number irreducible constituents of the Weil
module $X$ is equal to the number of orbits of $U$ acting on $V$.
This common number is:

\noindent $\bullet$ $q^\ell$ if $m=1$.

\noindent $\bullet$ $2q^{\ell}-1$ if $m=2$ and $h$ is of type $\{1,-\epsilon\}$.

\noindent  $\bullet$ $2(q^{\ell}+\cdots+q)+1$ if $m>2$ or $m=2$ and $h$ is of type $\{1,-1\}$.
\end{thm}

\noindent{\it Proof.} That the given quantities are the same
follows from Theorems \ref{comega} and \ref{mulfree}. The
remaining assertions follows by induction by means of Corollary
\ref{uva}.\qed

\section{Degrees of the irreducible constituents of $Top$}
\label{secdeg}

Let $\overline{R}=R/R\cap\i$ and
$\overline{A}=A/\i$. Then $\overline{A}$ inherits an
involution, say $a+\i\mapsto \widehat{a+\i}$, from $A$ whose fixed points form $\overline{R}$. Let
$\overline{V}=V/\i V$. Then $h$ gives rise to a
non-degenerate hermitian form $\overline{h}:\overline{V}\times
\overline{V}\to \overline{A}$, given by
$$
\overline{h}(v+\i V,
w+\i V)=h(v,w)+\i ,\quad v,w\in V.
$$

Note also that if $\ell$ is even then $R\cap\i=Rx^{\ell/2}$ and $\overline{A}$
is a ramified quadratic extension of $\overline{R}$. However, if $\ell$ is odd then
then $R\cap\i=Rx^{(\ell+1)/2}$ but
$\overline{A}$ is no longer a ramified quadratic extension of $\overline{R}$  since the classes of $y$ and $x^{(\ell-1)/2}$
multiply to 0 in $\overline{A}$.

Let $\overline{U}$ be unitary group associated to the hermitian
space $(\overline{V}, \overline{A})$. As shown in \cite{15}, Theorem 5.2, as
well as in \S 4 of \cite{8}, the canonical map $U\to \overline{U}$
is a group epimorphism.

Let $S$ be set of representatives of primitive vectors for the equivalence relation $\sim$ defined in \S\ref{sectop}.
Let $s\in S$. It follows from (\ref{indbv}), (\ref{k3}) and our description of $R\cap\i$ that
\begin{equation}
\label{indice}
[B(s):C(s)]=\begin{cases} 2q^{\ell/2} \text{ if }\ell\text{ is even},\\
2q^{(\ell-1)/2} \text{ if }\ell\text{ is odd}.\end{cases}
\end{equation}

Let $D(s)$ be the $\overline{U}$-pointwise stabilizer of $s+\i V$.
We see that $C(s)$ maps onto $D(s)$ under the epimorphism $U\to
\overline{U}$, so
$$
 [U:C(s)]=[\overline{U}:D(s)].
$$
Let $G$ be the group of linear characters of $N$ that are trivial on $N\cap (1+\i)$ and let $\phi\in
G$. It follows from Theorem \ref{dsm} that the degree of the irreducible constituent
$Top(\phi,s)$ of $Top$ is
$$\mathrm{deg}\, Top(\phi,s)=[\overline{U}:D(s)]/[B(s):C(s)],$$ where $[B(s):C(s)]$ is given in
(\ref{indice}).

The computation of $[\overline{U}:D(s)]$ is essentially equivalent to that of $|\overline{U}|$, a non-trivial problem
that will not be considered in this paper. However, if $\ell=1$ then $\overline{U}$ is an orthogonal group of rank $m$ over $F_q$,
whose order is well-known. This case is considered below.

\section{The case $\ell=1$}
\label{seclast}

We assume here that $\ell=1$. Thus $R=F_q$ and $A=F_q[y]$, $y^2=0$. Moreover, $\i=Ay$, so $R\cap\i=0$.
Thus the equivalence relation $\sim$ for primitive vectors considered in \S\ref{sectop} is given by: $u\sim v$ if $h(u,u)=h(v,v)$.
Hence, by Theorem \ref{crux}, a set $S$ of representatives for $\sim$ is a set of representatives for the $U$-orbits of $V\setminus yV$.

Let $\iota$ be the central involution of $U$ given by $v\mapsto -v$. Let $s\in S$. Then $B(v)=C(v)\times\{1,\iota\}$ by (\ref{indice}).
We may extend the linear character $\beta_v:C(s)\to\C^{\times}$ of \S\ref{con} to $B(s)$ in 2 ways, by
letting $\iota\mapsto \pm 1$. Let these extensions be denoted by
$\beta^{\pm}$.

The Weil module $X=Top\oplus Bot$ has dimension $q^m$, with $\dim\, Top=q^m-1$ and $\dim\, Bot=1$.
Let our transversal $T$ for $\i V$ in $V$ be symmetric,
in the sense that $v\in T$ if and only if $-v\in T$. In particular, $0\in T$. We see at once from (\ref{actu}) that $Bot=\C e_0$
is the trivial $U$-module. Let $Top^\pm$ be the eigenspaces for $\iota$ acting on $Top$. Then each of $Top^+$ and $Top^-$ has dimension $(q^m-1)/2$,
respectively spanned by $e_v+e_{-v}$ and $e_v-e_{-v}$ as $0\neq v$ runs through $T$.

We may also assume that $S$ is contained in $T$. For $s\in S$ let
$Top^\pm(s)$ be the $U$-submodule of $Top^\pm$ generated by
$e_s\pm e_{-s}$. Let $\delta\in F_q^{\times}\notin F_q^{\times 2}$. We then
have the following result.

\begin{thm} (a) $Top^+=\underset{v\in S}\oplus Top^+(s)$ and $Top^-=\underset{v\in S}\oplus Top^-(s)$.

(b) For $s\in S$, $Top^+(s)$ and $Top^-(s)$ are irreducible $U$-modules with characters $ind_{B(v)}^U\beta_v^+$ and $ind_{B(v)}^U\beta_v^-$.

(c) The Weil module $X$ is multiplicity free and has $2K+1$
irreducible constituents, namely the trivial $U$-module and the
$K$ constituents given in (a) for each of $Top^+$ and $Top^-$.
Here $K=q$ if $m>2$ or $m=2$ and $h$ is of type $\{1,-1\}$;
$K=q-1$ if $m=2$ and $h$ is of type $\{1,-\delta\}$; $K=(q-1)/2$
if $m=1$.
\end{thm}

Let $s\in S$. We next turn our attention to computing the degree
of $Top^\pm(s)$.

Let us adopt the notation of \S\ref{secdeg}. Then the degree of
$Top^\pm(s)$ is $[\overline{U}:D(s)]/2$. We wish to determine
$[\overline{U}:D(s)]$. Here $O=\overline{U}$ is the orthogonal
group of the non-degenerate symmetric bilinear form
$b=\overline{h}$ defined on the $m$-dimensional vector space
$W=\overline{V}$ over $F_q=\overline{A}$.

The classification of non-degenerate symmetric bilinear forms over
$F_q$ is completely analogous to that of hermitian forms stated in
Theorem \ref{xtipo} (see \cite{9}, Theorem 6.9). In particular,
$h_1,h_2$ are equivalent if and only if
$\overline{h_1},\overline{h_2}$ are equivalent. Since the
equivalence types of $h$ and $b$ determine each other, we will
refer to only that of $b$ from now on. It is well-known (see
\cite{15}, Theorem 4.1, for example) that two non-zero vectors in
$W$ of the same $b$-length are $O$-conjugate. Let $t=h(s,s)\in
F_q$. Thus, all we have to do is find the index in $O$ of the
stabilizer $S_t$ of a non-zero vector $u$ of length $t$.

If $m=1$ then $t$ must be a unit, $|S_t|=1$ and $|O|=2$. Thus
$\deg(Top^\pm(s))=1$, as expected since $U$ is abelian in this
case.

Suppose henceforth $m\geq 2$. When $m$ is even the orders of the
groups associated with the two inequivalent quadratic forms are
different.  The group $O_m(q)$ is associated with the type
$\{1,-1,\dots,1,-1\}$ (type 1) and $O_m(q,\delta)$ is associated
with $\{1,-1,\dots,1,-\delta\}$ (type $\delta$). When $m$ is odd
there is just one orthogonal group $O_m(q)$ up to isomorphism, but
there are still two inequivalent quadratic forms:
$\{1,-1,\dots,1,-1,-1\}$ (type $1$), and
$\{1,-1,\dots,1,-1,-\delta\}$ (type $\delta$).  We will freely use
the formulas for the orders of these groups as given in \cite{9},
Theorem 6.17.

\begin{theorem}
Suppose $\ell=1$ and write $m=2r$ or $m=2r+1$ depending on the
parity of $m$, where $r\geq 1$.

Let $s \in S$ and let $t=h(s,s)\in F_q$. The degree of either of
the irreducible constituents $Top^\pm(s)$ of the Weil character of
$U_m(A)$ is given by one of the following:

$$\deg(Top^{\pm}(s)) = \dfrac{q^r(q^r + 1)}{2} \text{ in the following cases:} $$

(a1) $m=2r+1$, $b$ of type 1, $-1 \in F_q^{\times 2}$, $t \in F_q^{\times 2}$,

(a2) $m=2r+1$, $b$ of type 1, $-1 \in F_q^{\times} \setminus F_q^{\times 2}$,
$t \in F_q^{\times} \setminus F_q^{\times 2}$,

(a3) $m=2r+1$, $b$ of type $\delta$, $-1 \in F_q^{\times 2}$, $t \in
F_q^{\times} \setminus F_q^{\times 2}$, and

(a4) $m=2r+1$, $b$ of type $\delta$, $-1 \in F_q^{\times} \setminus
F_q^{\times 2}$, $t \in F_q^{\times 2}$.

$$\deg(Top^{\pm}(s)) = \dfrac{q^r(q^r - 1)}{2} \text{ in the following cases:} $$

(b1) $m=2r+1$, $b$ of type 1, $-1 \in F_q^{\times 2}$, $t \in F_q^{\times}
\setminus F_q^{\times 2}$,

(b2) $m=2r+1$, $b$ of type 1, $-1 \in F_q^{\times} \setminus F_q^{\times 2}$,
$t \in F_q^{\times 2}$,

(b3) $m=2r+1$, $b$ of type $\delta$, $-1 \in F_q^{\times 2}$, $t \in
F_q^{\times 2}$, and

(b4) $m=2r+1$, $b$ of type $\delta$, $-1 \in F_q^{\times} \setminus
F_q^{\times 2}$, $t \in F_q^{\times} \setminus F_q^{\times 2}$.

$$\deg(Top^{\pm}(s)) = \dfrac{q^{r-1}(q^r + 1)}{2} \text{ if } $$

(c) $m=2r$, $b$ is of type $\delta$, and $t \in F_q^{\times}$.

$$\deg(Top^{\pm}(s)) = \dfrac{q^{r-1}(q^r - 1)}{2} \text{ if } $$

(d) $m=2r$, $b$ is of type 1, and $t \in F_q^{\times}$.

$$\deg(Top^{\pm}(s)) = \dfrac{(q^r-1)(q^{r-1} + 1)}{2} \text{ if }$$

(e) $m=2r$, $b$ is of type 1, and $t=0$.

$$\deg(Top^{\pm}(s)) = \dfrac{(q^r+1)(q^{r-1} - 1)}{2} \text{ if } $$

(f) $m=2r>2$ and $b$ is of type $\delta$, and $t=0$.

$$\deg(Top^{\pm}(s)) = \dfrac{(q^{2r}-1)}{2} \text{ if } $$

(g) $m=2r+1$, $b$ is of either type, and $t=0$.
\end{theorem}

\begin{proof}
We wish to find the index in $O$ of the stabilizer $S_t$ of a
vector $u=s + \i V\in W$ of length $b(u,u)=t$.

Suppose $t$ is a unit. Then $|S_t|=|O'|$, where $O'$ is the
orthogonal group of the orthogonal complement to $u$ in $W$.

When $m=2$ then $|S_t|=2$. Now if $b$ is of type $\{1,-1\}$ then
$O=O_2(q)$ has order $2(q-1)$, while if $b$ is of type
$\{1,-\delta\}$  then $O=O_2(q,\delta)$ has order $2(q+1)$. Thus
$$\deg(Top^{\pm}(s)) = \begin{cases} \frac{q-1}{2} & \text{ if $W$ has type $\{1,-1\}$, and } \\ \frac{q+1}{2} & \text{ if $W$ has type $\{1,-\delta\}$.} \end{cases}
$$

For $m \ge 3$, we need to keep track of the type of $b$ even when
$m$ is odd, as the Weil character depends on the form, not just
its isometry group.  We can alter $b$ in a manner that preserves
the determinant modulo squares. We may freely replace $t$ by $k^2
t$ for any $k\in F_q^{\times}$ as the stabilizers of $u$ and $ku$ are
identical.

Suppose $m=2r+1$ and $t$ is a square. Assume $b$ is of type
$\{1,-1,\dots,1,-1,-1\}$. If $-1$ is a square we may take $u$ to
be the last vector in the basis. Since the stabilizer of $u$ in
$O$ acts as the full symmetric group on the orthogonal complement
of $u$ in $W$, then $S_t$ is isomorphic to an orthogonal group for
a quadratic form of type $\{1,-1,\dots,1,-1\}$. Therefore,
$|S_t|=|O_{2r}(q)|$ and we have
$$[O:S_t]=\dfrac{|O_{2r+1}(q)|}{|O_{2r}(q)|} = q^r(q^r+1).$$
If $-1$ is not a square then $-\delta\in F_q^{\times 2}$. Convert the
type of $b$ to $\{1,-1,\dots,1,-\delta,-\delta\}$ and take $u$ to
be the last vector in this basis. Then $S_t$ is isomorphic to the
orthogonal group of type $\{1,-1,\dots,1,-\delta\}$, and we have
$$[O:S_t]=\dfrac{|O_{2r+1}(q)|}{|O_{2r}(q,\delta)|} = q^r(q^r-1).$$
Assume next $b$ has type $\{1,-1,\dots,1,-1,-\delta\}$. If $-1$ is
a square, we can take $u$ to be the second-to-last vector in this
basis, so $S_t \simeq O_{2r}(d,\delta)$. If $-1$ is not a square,
then $-\delta$ will be a square and we have $S_t \simeq
O_{2r}(q)$.

Using the same techniques as above, we can find the value of
$[O:S_t]$ when $m$ is odd and $t$ is not a square in each of the
cases of $b$ of type 1 or $\delta$, $-1$ a square or non-square.
We leave these details to the reader.

Now suppose $m=2r$ is even.  By making suitable choices for $u$
and then eliminating it as above, we find that $S_t \simeq
O_{2r-1}(q)$ in all cases. So the calculation of $[O:S_t]$ only
depends on which type of $b$ we started with.

Suppose next $t=0$. Then there is a basis $u,v,w_1,\dots,w_{m-2}$
of $W$ such that
$$b(u,u)=0=b(v,v),b(u,v)=1, b(u,w_i)=0=b(v,w_i)$$
and the stabilizer $S_0$ of $u$ in $O$ has order
$$
|S_0|=q^{m-2}|O'|,
$$
where $O'$ is the orthogonal group of rank $m-2$ associated to the
restriction $b'$ of $b$ to the span $W'$ of $w_1,\dots,w_{m-2}$
(see \cite{21}, pg. 72). In particular, since $\mathrm\{u,v\}$ is
a hyperbolic plane (and hence of type $\{1,-1\}$), the form $b'$
has the same type as $b$. Note that when $m=2$ this says that
$|S_0|=1$. If $m$ is odd, then
$$[O:S_0] = \dfrac{|O_m(q)|}{q^{m-2}|O_{m-2}(q)|} = q^{m-1}-1.$$
If $m$ is even and $b$ has type 1 (this is the only possibility
when $m=2$), then
$$[O:S_0] = \dfrac{|O_{2r}(q)|}{q^{2r-2}|O_{2(r-1)}(q)|} = (q^r-1)(q^{r-1}+1). $$
If $m$ is even and $W$ has type $\delta$, then
$$[O:S_0] = \dfrac{|O_{2r}(q,\delta)|}{q^{2r-2}|O_{2(r-1)}(q,\delta)|} = (q^r+1)(q^{r-1}-1).$$
This completes the proof of the theorem, since the degrees of the
corresponding irreducible characters are obtained simply by
dividing $[O:S_0]$ by $2$.
\end{proof}

\end{document}